\documentclass{article}
\usepackage{amsmath,amsfonts,amsthm,amssymb,amscd,color,xcolor}
\usepackage{mathrsfs}
\usepackage{hyperref}

\colorlet{darkblue}{blue!50!black}

\hypersetup{
    colorlinks,%
    citecolor=darkblue,%
    filecolor=red,%
    linkcolor=darkblue,%
    urlcolor=magenta,%
    pdfnewwindow=true,%
    pdfstartview={FitH}
}

\newcommand{\p}{\partial}
\newcommand{\e}{\varepsilon}

\newcommand{\R}{{\mathbb R}}

\newcommand{\Z}{{\mathbb Z}}

\newcommand{\dd}{{\textup d}}

\newcommand{\RR}{{\cal R}}

\newcommand{\XX}{{\cal X}}

\newcommand{\supp}{\mathop{\rm supp}\nolimits}

\newcommand{\esssup}{\mathop{\rm ess\,sup}\limits}

\theoremstyle{plain}

\newtheorem{theorem}{Theorem}[section]
\newtheorem{lemma}[theorem]{Lemma}
\newtheorem{proposition}[theorem]{Proposition}
\newtheorem{corollary}[theorem]{Corollary}
\theoremstyle{definition}
\newtheorem{definition}[theorem]{Definition}

\theoremstyle{remark}
\newtheorem{remark}[theorem]{Remark}

\numberwithin{equation}{section}

\begin{document}
\title{Global exponential stabilisation for the Burgers equation with localised control}
\author{Armen Shirikyan\footnote{Department of Mathematics, University of Cergy--Pontoise, CNRS UMR 8088, 2 avenue Adolphe Chauvin, 95302 Cergy--Pontoise, France; e-mail: \href{mailto:Armen.Shirikyan@u-cergy.fr}{Armen.Shirikyan@u-cergy.fr}}}

\date{}
\maketitle
\begin{abstract}
We consider the 1D viscous Burgers equation with a control localised in a finite interval. It is proved that, for any $\e>0$, one can find a time~$T$ of order $\log\e^{-1}$ such that any initial state can be steered to the $\e$-neighbourhood of a given trajectory at time~$T$. This property combined with an earlier result on local exact controllability shows that the Burgers equation is globally exactly controllable to trajectories in a finite time. We also prove that the approximate controllability to arbitrary targets does not hold even if we allow infinite time of control. 


\end{abstract}

\tableofcontents

\setcounter{section}{-1}

\section{Introduction}
\label{s1}
Let us consider the controlled Burgers equation on the interval $I=(0,1)$ with the Dirichlet boundary condition:
\begin{align}
\p_tu-\nu\p_x^2u+u\p_xu&=h(t,x)+\zeta(t,x), \label{1}\\ 
u(t,0)=u(t,1)&=0. \label{2}
\end{align}
Here $u=u(t,x)$ is an unknown function, $\nu>0$ is a parameter, $h$ is a fixed function, and~$\zeta$ is a control that is assumed to be localised in an interval $[a,b]\subset I$. As is known, the initial-boundary value problem for~\eqref{1} is well posed. Namely, if $h\in L_{\mathrm{loc}}^2(\R_+,L^2(I))$ and $\zeta\equiv0$, then, for any $u_0\in L^2(I)$, problem~\eqref{1}, \eqref{2} has a unique solution $u(t,x)$ that belongs to the space 
$$
\XX=\{u\in L_{\mathrm{loc}}^2(\R_+,H_0^1(I)):
\p_tu\in L_{\mathrm{loc}}^2(\R_+,H^{-1}(I))\}
$$
and satisfies the initial condition
\begin{equation} \label{3}
u(0,x)=u_0(x);
\end{equation}
see the end of this Introduction for the definition of functional spaces. Let us denote by~$\RR_t(u_0,h)$ the mapping that takes the pair~$(u_0,h)$ to the solution~$u(t)$ (with $\zeta\equiv0$). We wish to study the problem of controllability for~\eqref{1}. This question received great deal of attention in the last twenty years, and we now recall some  achievements related to our paper. 

One of the first results was obtained by Fursikov and Imanuvilov~\cite{FI-1995,FI1996}. They established the following two properties: 

\medskip
{\bf Local exact controllability.} {\sl Let $\hat u(t,x)$ be a trajectory of~\eqref{1}, \eqref{2} with $\zeta\equiv0$ and let $T>0$. Then there is $\e>0$ such that, for any $u_0\in H_0^1(I)$ satisfying the inequality $\|u_0-\hat u(0)\|_{H^1}\le\e$, one can find a control\footnote{We denote by~$J_T$ the time interval $[0,T]$.} $\zeta\in L^2(J_T\times I)$ supported in $J_T\times [a,b]$ for which $\RR_T(u_0,h+\zeta)=\hat u(T)$. Moreover, when~$T$ is fixed, the number~$\e$ can be chosen to be the same for all~$\hat u(0)$ and~$h$ varying in bounded subsets of the spaces~$H_0^1(I)$ and~$L^2(J_T\times I)$, respectively.}

\smallskip
{\bf Absence of approximate controllability.} 
{\sl For any $u_0\in L^2(I)$ and any positive numbers~$T$ and~$R$, one can find $\hat u\in L^2(I)$ such that, for any control~$\zeta\in L^2(J_T\times I)$ supported by $J_T\times [a,b]$, we have} 
\begin{equation} \label{0.4}
\|\RR_T(u_0,h+\zeta)-\hat u\|\ge R.
\end{equation}

\medskip
\noindent
These results were extended and developed in many works. In particular, Glass and Guererro~\cite{GG-2007}  and~\cite{leautaud-2012} proved global exact boundary controllability to constant states, Coron~\cite{coron-2007} and~Fern\'andez-Cara--Guererro~\cite{FG-2007} established some estimates for the time and cost of control, and Chapouly~\cite{chapouly-2009} (see also Marbach~\cite{marbach-2014}) proved global exact controllability to trajectories with two boundary and one distributed controls, provided that $h\equiv0$.  A large number of works were devoted to the investigation of similar question for other, more complicated equations of fluid mechanics; see the references in~\cite{fursikov2000,coron2007}.

\smallskip
In view of the above-mentioned properties, two natural questions arise: 
\begin{itemize}
\item
does the {\it exact controllability to trajectories\/} hold  for arbitrary initial conditions and nonzero right-hand sides\,?  
\item
does the {\it approximate controllability\/} hold if we allow a sufficiently large time of control\,?
\end{itemize}
It turns out that the answer to the first question is positive, provided that the time of control is sufficiently large, whereas the answer to the second question is negative. Namely, the main results of this paper combined with the above-mentioned property of local exact controllability to trajectories imply the following theorem.\footnote{See the Notation below for definition of the spaces used in the statement.}

\medskip
{\bf Main Theorem.}
{\it Suppose that $\nu>0$, $h\in (H_{\mathrm{ul}}^1\cap L^\infty)(\R_+\times I)$, and $[a,b]\subset I$. Then the following two assertions hold.
\begin{itemize}
\item[\bf(a)]
There exists $T>0$ such that, for any $u_0,\hat u_0\in L^2(I)$ one can find a control $\zeta\in L^2(J_T\times I)$ supported by $J_T\times [a,b]$ for which
\begin{equation} \label{11}
\RR_T(u_0,h+\zeta)=\RR_T(\hat u_0,h). 
\end{equation}
\item[\bf(b)] 
For any positive numbers~$T_0$ and~$R$, one can find $\hat u\in L^2(I)$ such that, for any $u_0\in L^2(I)$ and any control $\zeta\in L_{\mathrm{loc}}^2(\R_+\times I)$ supported by $\R_+\times[a,b]$, inequality~\eqref{0.4} holds for $T\ge T_0$. 
\end{itemize}
}

\smallskip
Let us mention that the result about exact controllability to trajectories remain valid for a much larger class of scalar conservation laws in higher dimension. This question will be addressed in a subsequent publication. 

The rest of the paper is organised as follows. In Section~\ref{s2}, we formulate a result on exponential stabilisation to trajectories, outline the scheme of its proof, and derive  assertion~(a) of the Main Theorem. Section~\ref{s3} is devoted to some preliminaries about the Burgers equation. In Section~\ref{s4}, we present the details of the proof of exponential stabilisation and establish property~(b) of the Main Theorem. Finally, the Appendix gathers the proofs of some auxiliary results.

\medskip
{\bf Acknowledgements}. 
This research was carried out within the MME-DII Center of Excellence (ANR-11-LABX-0023-01) and supported by the RSF grant 14-49-00079. 

\subsubsection*{Notation}
Let $I=(0,1)$, $J_T=[0,T]$, $\R_+=[0,+\infty)$, and $D_T=(T,T+1)\times I$. We use the following function spaces.

\smallskip
\noindent
$L^p(D)$ and $H^s(D)$ are the usual Lebesgue and Sobolev spaces, endowed with natural norms~$\|\cdot\|_{L^p}$ and~$\|\cdot\|_s$, respectively. In the case $p=2$ (or $s=0$), we write~$\|\cdot\|$ and denote by~$(\cdot,\cdot)$ the corresponding scalar product. 

\smallskip
\noindent
$C^\gamma(D)$ denotes the space of H\"older-continuous functions with exponent $\gamma\in(0,1)$. 

\smallskip
\noindent
$H_{\mathrm{loc}}^s(D)$ is the space of functions $f\!:D\to\R$ whose restriction to any bounded open subset $D'\subset D$ belongs to~$H^s(D')$. 

\smallskip
\noindent
$H_{\mathrm{ul}}^s(\R_+\times I)$ stands for the space of functions $u\in H_{\mathrm{loc}}^s(\R_+\times I)$ satisfying the condition
$$
\|u\|_{H_{\mathrm{ul}}^s}:=\sup_{T\ge0}\|u\|_{H^s(D_T)}<\infty.
$$

\noindent
Very often, the context implies the domain on which a functional space is defined, and in this case we omit it from the notation. For instance, we write $L^2$, $H^s$, etc. 

\smallskip
\noindent
$L^p(J,X)$ is the space of Borel-measurable functions $f:J\to X$ (where $J\subset \R$ is a closed interval and~$X$ is a separable Banach space) such that
$$
\|f\|_{L^p(J,X)}=\biggl(\int_J\|f(t)\|_X^p\dd t\biggr)^{1/p}<\infty.
$$
In the case $p=\infty$, this condition should be replaced by 
$$
\|f\|_{L^\infty(J,X)}=\esssup_{t\in J}\|f(t)\|_X<\infty.
$$

\noindent
$W^{k,p}(J,X)$ is the space of functions $f\in L^p(J,X)$ such that $\p_t^j f\in L^p(J,X)$ for $1\le j\le k$, and if~$J$ is unbounded, then $W_{\mathrm{loc}}^{k,p}(J,X)$ is the space of functions whose restriction to any bounded interval $J'\subset J$ belongs to $W^{k,p}(J',X)$.

\smallskip
\noindent
$C(J,X)$ is the space of continuous functions $f:J\to X$. 

\smallskip
\noindent
$B_X(a,R)$ denotes the closed ball in~$X$ of radius~$R\ge0$ centred at~$a\in X$. In the case $a=0$, we write~$B_X(R)$.

\section{Exponential stabilisation to trajectories}
\label{s2}
Let us consider problem~\eqref{1}, \eqref{2}, in which  $\nu>0$ is a fixed parameter, $h(t,x)$ is a given function belonging to $H_{\mathrm{ul}}^1\cap L^\infty$ on the domain $I\times\R_+$, and $\zeta$ is a control taking values in the space of functions in~$L^2(I)$ with support in a given interval $[a,b]\subset I$. Recall that~$\RR_t(u_0,h+\zeta)$ stands for the value of the solution for \eqref{1}--\eqref{3} at time~$t$. The following theorem is the main result of this paper. 

\begin{theorem} \label{t2.1}
Under the above hypotheses, there exist positive numbers $C$ and~$\gamma$ such that, given arbitrary initial data $u_0, \hat u_0\in L^2(I)$, one can find a piecewise continuous control $\zeta:\R_+\to H^1(I)$ supported in $\R_+\times[a,b]$ for which 
\begin{equation} \label{5}
\|\RR_t(u_0,h+\zeta)-\RR_t(\hat u_0,h)\|_1+\|\zeta(t)\|_1
\le Ce^{-\gamma t}\min\bigl(\|u_0-\hat u_0\|_{L^1}^{2/5},1\bigr), \quad t\ge1.
\end{equation}
Moreover, the  control~$\zeta$ regarded as a function of time may have discontinuities only at positive integers. 
\end{theorem}

As was mentioned in the Introduction, this theorem combined with the Fursikov--Imanuvilov result on local exact controllability (see~\cite[Section~I.6]{FI1996}) implies that the Burgers equation is exactly controllable to trajectories in a finite time independent of the initial data. Indeed, for any $\hat u_0\in L^2(I)$ the trajectory~$\hat u(t)=\RR_t(\hat u_0,h)$ is bounded in~$H_0^1(I)$ for $t\ge1$. In view of the local exact controllability, one can find $\e>0$ such that, if $v_0\in H_0^1(I)$ satisfies the inequality $\|v_0-\hat u(T)\|_1\le\e$ for some $T\ge1$, then there is a control $\zeta\in L^2(D_T)$ supported in $[T,T+1]\times [a,b]$ such that $v(T+1)=\hat u(T+1)$, where $v(t,x)$ stands for the solution of \eqref{1}, \eqref{2} issued from~$v_0$ at time $t=T$. Due to~\eqref{5}, there is $T_\e>0$ such that, for any $u_0,\hat u_0\in L^2(I)$, one can find a piecewise continuous control $\zeta:J_{T_\e}\to H^1(I)$ supported in $J_{T_\e}\times [a,b]$ for which 
$$
\|\RR_{T_\e}(u_0,h+\zeta)-\hat u(T_\e)\|_1\le\e.
$$ 
Applying the above result on local exact controllability to $v_0=\RR_{T_\e}(u_0,h+\zeta)$, we arrive at assertion~(a) of the Main Theorem stated in the Introduction. 

\smallskip
We now outline the main steps of the proof of Theorem~\ref{t2.1}, which is given in Section~\ref{s4}.  It is based on a comparison principle for nonlinear parabolic equations and the Harnack inequality. 

\subsubsection*{Step~1: Reduction to bounded regular initial data}
We first prove that it suffices to consider the case of $H^2$-smooth initial conditions with norm bounded by a fixed constant. Namely, let $V:=H_0^1\cap H^2$, and given a number $T>0$, let us define the functional space 
\begin{equation} \label{1.2}
\XX_T=L^2(J_T,H_0^1)\cap W^{1,2}(J_T,H^{-1}). 
\end{equation}
We have the following result providing a universal bound for solutions of~\eqref{1}, \eqref{2} at any positive time. 

\begin{proposition} \label{p1.2}
Let $h\in (H^1\cap L^\infty)(J_T\times I)$ for some $T>0$ and let $\nu>0$. Then there is $R>0$ such that any solution $u\in\XX_T$ of~\eqref{1} with $\zeta\equiv0$ satisfies the inclusion $u(t)\in V$ for $0<t\le T$ and the inequality
\begin{equation} \label{1.3}
\|u(T)\|_2\le R.
\end{equation}
\end{proposition}

Thus, if $h\in H_{\mathrm{ul}}^1\cap L^\infty$ is fixed, then, for any initial data $u_0,\hat u_0\in L^2(I)$, we have 
$$
\|\RR_1(u_0,h)\|_2\le R, \quad \|\RR_1(\hat u_0,h)\|_2\le R,
$$
where $R$ is the constant in Proposition~\ref{p1.2} with $T=1$. 
Furthermore, in view of the contraction of the $L^1$-norm for the difference of two solutions (cf.\ Proposition~\ref{p3.3} below), we have 
$$
\|\RR_1(u_0,h)-\RR_1(\hat u_0,h)\|_{L^1}\le \|u_0-\hat u_0\|_{L^1}.
$$
Hence, to prove Theorem~\ref{t2.1}, it suffices to establish the inequality in~\eqref{5} for $t\ge0$ and any initial data $u_0,\hat u_0\in B_V(R)$. 

\subsubsection*{Step~2: Interpolation}
Let us fix two initial conditions $u_0,\hat u_0\in B_V(R)$. Suppose we have constructed a control $\zeta(t,x)$ supported in~$\R_+\times [a,b]$ such that, for all $t\ge0$,
\begin{align}
\|\RR_t(u_0,h+\zeta)\|_2+\|\RR_t(\hat u_0,h)\|_2&\le C_1,\label{1.4}\\
\|\RR_t(u_0,h+\zeta)-\RR_t(\hat u_0,h)\|_{L^1}&
\le C_2e^{-\alpha t}\|u_0-\hat u_0\|_{L^1},\label{1.5}
\end{align}
where $C_1$, $C_2$, and~$\alpha$ are positive numbers not depending on~$u_0$, $\hat u_0$, and~$t$. In this case, using the interpolation inequality (see Section~15.1 in~\cite{BIN1979})
\begin{equation} \label{1.6}
\|v\|_1\le C_3\|v\|_{L^1}^{2/5}\|v\|_2^{3/5}, \quad v\in H^2(I),
\end{equation}
we can write
\begin{equation} \label{1.7}
\|\RR_t(u_0,h+\zeta)-\RR_t(\hat u_0,h)\|_{1}
\le C_4e^{-\gamma t}\|u_0-\hat u_0\|_{L^1}^{2/5},
\end{equation}
where $\gamma=\frac{2\alpha}{5}$, and~$C_4>0$ does not depend on~$u_0$, $\hat u_0$, and~$t$. This implies the required inequality for the first term on the left-hand side of~\eqref{5}. An estimate for the second term will follow from the construction; see relations~\eqref{1.16} and~\eqref{1.13} below.

\subsubsection*{Step~3: Main auxiliary result}
Let us take two initial data $v_0,\hat u_0\in B_V(R)$ and consider the difference~$w$ between the corresponding solutions of problem~\eqref{1}--\eqref{3} with $\zeta\equiv0$; that is, $w=v-\hat u$, where $v(t)=\RR_t(v_0,h)$ and $\hat u(t)=\RR_t(\hat u_0,h)$. It is straightforward to check that~$w$ satisfies the linear equation
\begin{equation} \label{A.31}
\p_t w-\nu\p_x^2w+\p_x\bigl(a(t,x)w\bigr)=0,
\end{equation}
where $a=\frac12(v+\hat u)$. The following proposition is the key point of  our construction. 

\begin{proposition} \label{p1.4}
Let positive numbers~$\nu$, $T$, $\rho$, and $s<1$ be fixed, and let~$a(t,x)$ be a function such that
\begin{equation} \label{1.09}
\|a\|_{C^s(J_T\times I)}+\|\p_xa\|_{L^\infty(J_T\times I)}\le\rho.
\end{equation}
Then, for any closed interval $I'\subset I$, there are positive numbers~$\e$ and $q<1$, depending only on $\nu$, $T$, $\rho$, $s$, and~$I'$, such that any solution~$w\in\XX_T$ of Eq.~\eqref{A.31} satisfies one of the inequalities
\begin{equation} \label{1.10}
\|w(T)\|_{L^1}\le q\,\|w(0)\|_{L^1}\quad\mbox{or} \quad 
\|w(T)\|_{L^1(I')}\ge \e\,\|w(0)\|_{L^1}. 
\end{equation}
\end{proposition}

In other words, for the difference of any two solutions, either the $L^1$-norm  undergoes a strict contraction or a non-trivial mass is concentrated on~$I'$. In both cases, we can modify the difference between the reference and uncontrolled solutions in the neighbourhood of~$I'$ so that the resulting function is a solution to the controlled problem, and the $L^1$-norm of the difference decreases exponentially with time. We now describe this idea in more detail. 

\subsubsection*{Step~4: Description of the controlled solution}
Let us fix a closed interval $I'\subset (a,b)$ and choose two functions $\chi_0\in C^\infty(\bar I)$ and $\beta\in C^\infty(\R)$ such that
\begin{align}
0\le\chi_0(x)\le1&\mbox{ for $x\in I$},&  
\chi_0(x)&=0\mbox{ for $x\in I'$}, &
\chi_0(x)&=1\mbox{ for $x\in I\setminus [a,b]$}, 
\label{1.11}\\
0\le\beta(t)\le1&\mbox{ for $t\in\R$},&  \beta(t)&=0\mbox{ for $t\le\tfrac12$}, &
 \beta(t)&=1\mbox{ for $t\ge1$}.
\label{1.12}
\end{align}
Let us set $\chi(t,x)=1-\beta(t)(1-\chi_0(x))$. Given $u_0,\hat u_0\in B_V(R)$, we denote by~$\hat u(t,x)$ the reference trajectory and define a controlled solution~$u(t,x)$ of~\eqref{1} consecutively on intervals $[k,k+1]$ with $k\in\Z_+$ by the following rules:
\begin{itemize}
\item[\bf(a)]
\sl if $u(t)$ is constructed on $[0,k]$, then we denote by $v(t,x)$ the solution issued from~$u(k)$ for problem~\eqref{1}, \eqref{2} on~$[k,k+1]$ with $\zeta\equiv0\,;$ 
\item[\bf(b)]
for any odd integer $k\in\Z_+$, we set
\begin{equation} \label{1.16}
u(t,x)=v(t,x)\quad \mbox{for $(t,x)\in [k,k+1]\times I$}.
\end{equation}
\item[\bf(c)]
for any even integer $k\in\Z_+$, we set
\begin{equation} \label{1.13}
u(t,x)=\hat u(t,x)+\chi(t-k,x)\bigl(v(t,x)-\hat u(t,x)\bigr)\quad 
\mbox{for $(t,x)\in [k,k+1]\times I$}.
\end{equation}
\end{itemize}
It is not difficult to check that $u(t,x)$ is a solution of problem~\eqref{1}, \eqref{2}, in which~$\zeta$ is supported by~$\R_+\times[a,b]$. Moreover, it will follow from Proposition~\ref{p1.4} that, for any even integer~$k\ge0$, we have 
\begin{equation} \label{1.14}
\|u(k+1)-\hat u(k+1)\|_{L^1}\le \theta\,\|u(k)-\hat u(k)\|_{L^1},
\end{equation}
where $\theta<1$ does not depend on~$\hat u_0$, $u_0$, and~$k$. On the other hand, the contraction of the $L^1$-norm between solutions of~\eqref{1} implies that 
\begin{equation} \label{1.15}
\|u(t)-\hat u(t)\|_{L^1}\le \|u([t])-\hat u([t])\|_{L^1}
\quad\mbox{for any $t\ge0$},
\end{equation}
where $[t]$ stands for the largest integer not exceeding~$t$. These two inequalities give~\eqref{1.5}. The uniform bounds~\eqref{1.4}  for the $H^2$-norm will follow from regularity of solutions for problem~\eqref{1}, \eqref{2}. 

\section{Preliminaries on the Burgers equation}
\label{s3}
In this section, we establish some properties of the Burgers equation. They are well known, and their proofs can be found in the literature in more complicated situations. However, for the reader's convenience, we outline some of those proofs in the Appendix to make the presentation self-contained. In this section, when talking about Eq.~\eqref{1}, we always assume that $\zeta\equiv0$.

\subsection{Maximum principle and regularity of solutions}
In this subsection, we discuss the well-posedness of the initial-boundary value problem for the Burgers equation. This type of results are very well known, and we only outline their proofs in the Appendix. Recall that $V=H_0^1\cap H^2$, and the space~$\XX$ was defined in the Introduction.  

\begin{proposition} \label{p3.1}
Let $u_0\in L^2(I)$ and $h\in L_{\mathrm{loc}}^1(\R_+,L^2(I))$. Then problem~\eqref{1}--\eqref{3} has a unique solution $u\in\XX$. Moreover, the following two properties hold.

\smallskip
\noindent
{\bf $\boldsymbol{L^\infty}$ bound}. 
If $h\in L_{\mathrm{loc}}^\infty(\R_+\times I)$ and $u_0\in L^\infty(I)$, then $u\in L_{\mathrm{loc}}^\infty(\R_+\times I)$. 

\smallskip
\noindent
{\bf Regularity}. If, in addition, $u_0\in V$ and $h\in H_{\mathrm{loc}}^1(\R_+\times I)$, then 
\begin{equation} \label{3.1}
u\in L_{\mathrm{loc}}^2(\R_+,H^3)\cap W_{\mathrm{loc}}^{1,2}(\R_+,H_0^1)\cap W_{\mathrm{loc}}^{2,2}(\R_+,H^{-1}). 
\end{equation}
\end{proposition}

Let us note that, if $u_0$ is only in the space~$L^2(I)$, then the conclusions about the~$L^\infty$ bound and the regularity remain valid on the half-line~$\R_\tau:=[\tau,+\infty)$ for any $\tau>0$. To see this, it suffices to remark that any solution $u\in \XX$ of~\eqref{1}, \eqref{2} satisfies the inclusion $u(\tau)\in H_0^1\cap H^2$ for almost every $\tau>0$. For any such $\tau>0$, one can apply Proposition~\ref{p3.1} to the half-line~$\R_\tau$ and conclude that the inclusions mentioned there are true with~$\R_+$ replaced by~$\R_\tau$. 

\subsection{Comparison principle}
The Burgers equation possesses a very strong dissipation property due to the nonlinear term. To state and prove the corresponding result, we need the concept of sub- and super-solution for Eq.~\eqref{1} with $\zeta\equiv0$. Let us fix $T>0$ and, given an interval $I'\subset I$, define\footnote{Note that, in contrast to~$\XX_T$, we do not require the elements of~$\XX_T(I')$ to vanish on~$\p I'$.}
$$
\XX_T(I')=L^2(J_T,H^1(I'))\cap W^{1,2}(J_T,H^{-1}(I')). 
$$

\begin{definition}
A function $u^+\in \XX_T(I')$ is called a {\it super-solution\/} for~\eqref{1} if
\begin{equation}
\int_0^T\bigl((\p_tu,\varphi)+(\nu\p_xu-\tfrac12u^2,\p_x\varphi)\bigr)\dd t
\ge\int_0^T(h,\varphi)\,\dd t,  \label{2.2}
\end{equation}
where $\varphi\in L^\infty(J_T,L^2(I'))\cap L^2(J_T,H_0^1(I'))$ is an arbitrary non-negative function. The concept of a {\it sub-solution\/} is defined similarly, replacing~$\ge$ by~$\le$.
\end{definition}

A proof of the following result can be found in Section~2.2 of~\cite{AL-1983} for a more general problem; for the reader's convenience, we outline it in the Appendix. 

\begin{proposition} \label{p3.2}
Let $h\in L^1(J_T,L^2)$, and let functions~$u^+$ and~$u^-$ belonging to~$\XX_T(I')$ be, respectively, super- and sub-solutions for~\eqref{1} such that\,\footnote{It is not difficult to see that the restrictions of the elements of~$\XX_T(I')$ to the straight lines $t=t_0$ and $x=x_0$ are well defined.}
\begin{equation}\label{2.1}
u^+(t,x)\ge u^-(t,x)\quad \mbox{for $t=0$, $x\in I'$ and $t\in[0,T]$, $x\in\p I'$},
\end{equation}
where the inequality holds almost everywhere. Then, for any $t\in J_T$, we have
\begin{equation} \label{2.3}
u^+(t,x)\ge u^-(t,x)\quad\mbox{for a.e.\ $x\in I'$}. 
\end{equation}
\end{proposition}

We now derive an a priori estimate for solutions of~\eqref{1}, \eqref{2}. 

\begin{corollary} \label{c2.3}
Let $u_0\in L^\infty$ and $h\in L^\infty(J_T\times I)$ for some $T>0$. Then the solution of problem~\eqref{1}--\eqref{3} with $\zeta\equiv0$ satisfies the inequality
\begin{equation} \label{2.4}
\|u(T,\cdot)\|_{L^\infty}\le C,
\end{equation}
where $C>0$ is a number continuously depending only on~$\|h\|_{L^\infty}$ and~$T$. 
\end{corollary}

\begin{proof}
We follow the argument used in the proof of Lemma~9 in~\cite[Section~2.1]{coron-2007}. Given $\e>0$ and $u_0\in L^\infty(I)$, we set 
$$
B_\e=1+\|h\|_{L^\infty}^{1/3}(T+\e)^{2/3}, \quad L=\|u_0\|_{L^\infty}. 
$$
It is a matter of a simple calculation to check that the functions
$$
u_\e^+(t,x)=\frac{B_\e(B_\e+x)+L\e}{t+\e},\quad 
u_\e^-(t,x)=-\frac{B_\e(B_\e-x)+L\e}{t+\e}
$$
are, respectively, super- and sub-solutions for~\eqref{1} on the interval~$J_T$ such that 
$$
u_\e^+(t,x)\ge u(t,x)\ge u_\e^-(t,x)\quad 
\mbox{for $t=0$, $x\in I$ and $t\in[0,T]$, $x=0$ or~$1$}. 
$$
Applying Proposition~\ref{p3.2}, we conclude that 
$$
u_\e^+(T,x)\ge u(T,x)\ge u_\e^-(T,x)\quad\mbox{for a.e.~$x\in I$}. 
$$
Passing to the limit as $\e\to0^+$, we arrive at~\eqref{2.4} with $C=T^{-1}B_0(B_0+1)$. 
\end{proof}

\subsection{Contraction of the $L^1$-norm of the difference of solutions}
It is a well known fact that the resolving operator for~\eqref{1}, \eqref{2} regarded as a nonlinear mapping in the space~$L^2(I)$ is locally Lipschitz. The following result shows that it is a contraction for the norm of~$L^1(I)$. 

\begin{proposition} \label{p3.3}
Let $u,v\in \XX$ be two solutions of Eq.~\eqref{1}, in which $\zeta\equiv0$ and $h\in L_{\mathrm{loc}}^1(\R_+,L^2)$. Then 
\begin{equation} \label{21}
\|u(t)-v(t)\|_{L^1}\le \|u(s)-v(s)\|_{L^1}\quad\mbox{for any $t\ge s\ge 0$}. 
\end{equation}
\end{proposition}
Inequality~\eqref{21} follows from the maximum principle for linear parabolic PDE's, and more general results can be found in Sections~3.2 and~3.3 of~\cite{hormander1997}. A simple proof of Proposition~\ref{p3.3} is given in Section~\ref{A3}. 

\subsection{Harnack inequality}
\label{s2.4}
Let us consider the linear homogeneous equation~\eqref{A.31}. The following result is a particular case of the Harnack inequality established in~\cite[Theorem~1.1]{KS-1980} (see also Section~IV.2 in~\cite{krylov1987}). 

\begin{proposition} \label{p2.6}
Let a closed interval $K\subset I$ and positive numbers~$\nu$ and~$T$ be fixed. Then, for any $\rho>0$ and $T'\in(0,T)$, one can find $C>0$ such that the following property holds: if $a(t,x)$ satisfies the inequality
\begin{equation} \label{1.9}
\|a\|_{L^\infty(J_T\times I)}+\|\p_xa\|_{L^\infty(J_T\times I)}\le\rho,
\end{equation}
then for any non-negative solution  $w\in L^2(J_T,H^3\cap H_0^1)\cap W^{1,2}(J_T,H_0^1)$ of~\eqref{A.31} we have
\begin{equation} \label{2.6}
\sup_{x\in K} w(T',x)\le C\inf_{x\in K}w(T,x). 
\end{equation}
\end{proposition}

\section{Proof of the main results}
\label{s4}

In this section, we give the details of the proof of Theorem~\ref{t2.1} (Sections~\ref{s3.1}--\ref{s3.3}) and establish assertion~(b) of the Main Theorem stated in the Introduction (Section~\ref{s3.4}). 

\subsection{Reduction to smooth initial data}
\label{s3.1}
Let us prove Proposition~\ref{p1.2}. Fix arbitrary numbers $T_1<T_2$ in the interval $(0,T)$. By Proposition~\ref{p3.1} and the remark following it, for any $\tau>0$ we have 
\begin{equation} \label{3.01}
u\in L^\infty(J_{\tau,T}\times I)\cap 
L^2(J_{\tau,T},H^3)\cap W^{1,2}(J_{\tau,T},H_0^1)\cap W^{2,2}(J_{\tau,T},H^{-1}),
\end{equation}
where $J_{\tau,T}=[\tau,T]$. Applying Corollary~\ref{c2.3}, we see that 
\begin{equation} \label{3.02}
\|u(t,\cdot)\|_{L^\infty}\le C\quad\mbox{for $T_1\le t\le T$}. 
\end{equation}
Furthermore, it follows from~\eqref{3.01} that $u(t)$ is a continuous function of~$t\in (0,T]$ with range in~$V$. Thus, it remains to establish inequality~\eqref{1.3} with a universal constant~$R$. The proof of this fact can be carried out by a standard argument based on multipliers technique (e.g., see the proof of  Theorem~2 in~\cite[Section~I.6]{BV1992} dealing with the 2D Navier--Stokes system). Therefore, we confine ourselves to outlining the main steps. Until the end of this subsection, we deal with Eq.~\eqref{1} in which~$\zeta\equiv0$ and denote by~$C_i$ unessential positive numbers not depending~$u$. 

\medskip
{\it Step~1: Mean $H^1$-norm}.
Taking the scalar product of~\eqref{1} with~$2u$ and performing usual transformations, we derive 
$$
\p_t\|u\|^2+2\nu\|\p_xu\|^2=2(h,u)\le \nu\|\p_xu\|^2+\nu^{-1}\|h\|^2. 
$$
Integrating in time and using~\eqref{3.02} with $t=T_1$, we  obtain
\begin{equation} \label{3.03}
\int_{T_1}^T\|\p_xu\|^2\dd t\le \nu^{-1}\|u(T_1)\|^2+\nu^{-2}\int_{T_1}^T\|h\|^2\dd t\le C_1. 
\end{equation}

\smallskip
{\it Step~2: $H^1$-norm and mean $H^2$-norm}.
Let us take the scalar product of~\eqref{1} with $-2(t-T_1)\p_x^2u$: 
\begin{multline*}
\p_t\bigl((t-T_1)\|\p_xu\|^2\bigr) -\|\p_x u\|^2+2\nu(t-T_1)\|\p_x^2u\|^2
= 2(t-T_1)(u\p_xu-h,\p_x^2u)\\
\le 2(t-T_1)\bigl(\|u\|_{L^\infty}\|\p_xu\|+\|h\|\bigr)\|\p_x^2u\|. 
\end{multline*}
Integrating in time and using~\eqref{3.02} and~\eqref{3.03}, we obtain
\begin{equation} \label{3.04}
\|u(t)\|_1+\int_{T_2}^t\|u(r)\|_2^2\dd r\le C_2\quad\mbox{for $T_2\le t\le T$}.
\end{equation}
Using~\eqref{1}, we also derive the following estimate for~$\p_tu$:
\begin{equation} \label{3.05}
\int_{T_2}^T\|\p_tu\|^2\dd t\le C_3. 
\end{equation}

\smallskip
{\it Step~3: $L^2$-norm of the time derivative}.
Taking the time derivative of~\eqref{1}, we obtain the following equation for $v=\p_tu$:
$$
\p_tv-\nu\p_x^2v+v\p_xu+u\p_xv=\p_th.
$$
Taking the scalar product with $2(t-T_2)v$, we derive
\begin{align*}
\p_t\bigl((t-T_2)\|v\|^2\bigr) -\|v\|^2+2\nu(t-T_2)\|\p_xv\|^2
&= 2(t-T_2)(\p_th-u\p_xv-v\p_xu,v)\\
&\le 2(t-T_2)\bigl(\|\p_th\|+3\|u\|_{L^\infty}\|\p_xv\|\bigr)\|v\|. 
\end{align*}
Integrating in time and using~\eqref{3.02} and~\eqref{3.05}, we obtain
\begin{equation} \label{3.06}
\|v(T)\|\le C_4. 
\end{equation}

\smallskip
{\it Step~4: $H^2$-norm}.
We now rewrite~\eqref{1} in the form
\begin{equation} \label{3.07}
\nu\p_x^2u=f(t):=v+u\p_xu-h. 
\end{equation}
In view of~\eqref{3.04} and~\eqref{3.06}, we have $\|f(T)\|\le C_5$. Combining this with~\eqref{3.07}, we arrive at the required inequality~\eqref{1.3}. 

\begin{remark} \label{r3.1}
The argument given above shows that, under the hypotheses of Proposition~\ref{p1.2}, if $u_0\in B_V(\rho)$, then $\|R_t(u_0,h)\|_2\le R$ for all $t\ge0$, where $R>0$ depends only on~$h$, $\nu$, and~$\rho$. Moreover, similar calculations enable one to prove that, for any $t>0$, the resolving operator $\RR_t(u_0,h)$ regarded as a function of~$u_0$ is uniformly Lipschitz continuous from any ball of~$L^2$ to~$H^2$, and the corresponding Lipschitz constant can be chosen to be the same for $T^{-1}\le t\le T$, where $T>1$ is an arbitrary number. 
\end{remark}

\subsection{Proof of the main auxiliary result}
\label{s3.2}
In this subsection, we prove Proposition~\ref{p1.4}. In doing so, we fix parameter~$\nu>0$ and do not follow the dependence of various quantities on it. 

\medskip
{\it Step~1}. We begin with the case of non-negative solutions. Namely, we prove that, given $q\in(0,1)$, one can find $\delta=\delta(I',T,q,\rho)>0$ such that, if $w\in\XX_T$ is a non-negative solution of~\eqref{A.31}, then either the first inequality in~\eqref{1.10} holds, or 
\begin{equation} \label{3.08}
\inf_{x\in I'}w(T,x)\ge \delta \|w(0)\|_{L^1}. 
\end{equation}
To this end, we shall need the following lemma, established at the end of this subsection. 

\begin{lemma} \label{l3.1}
For any $0<\tau<T$ and $\rho>0$, there is $M>0$ such that, if~$w\in\XX_T$ is a solution of Eq.~\eqref{A.31} with a function~$a(t,x)$ satisfying~\eqref{1.09}, then 
\begin{equation}  \label{3.09}
\sup_{(t,x)\in[\tau,T]\times I}|w(t,x)|\le M\|w(0)\|_{L^1}. 
\end{equation}
\end{lemma}

In view of linearity, we can assume without loss of generality that~$\|w(0)\|_{L^1}=1$. Let us choose a closed interval $K\subset I$ containing~$I'$ such that 
\begin{equation}  \label{3.010}
|I\setminus K|\le\frac{q}{2M},
\end{equation}
where $|\Gamma|$ denotes the Lebesgue measure of a set $\Gamma\subset\R$, and $M>0$ is the constant in~\eqref{3.09} with $\tau=2T/3$. By Proposition~\ref{p3.1} and the remark following it, the function~$w$ satisfies the hypotheses of Proposition~\ref{p2.6}. Therefore, by the Harnack inequality~\eqref{2.6}, we have 
\begin{equation} \label{3.011}
\sup_{x\in K}w(2T/3,x)\le C\inf_{x\in K}w(T,x),
\end{equation}
where $C>0$ depends only on~$T$, $K$, and~$\rho$. Let us set $\delta=\frac{q}{2C|K|}$ and suppose that~\eqref{3.08} is not satisfied. In this case, using~\eqref{3.09}--\eqref{3.011} and the contraction of the $L^1$-norm of solutions for~\eqref{A.31} (see  Remark~\ref{r4.2}), we derive
\begin{align*}
\|w(T)\|_{L^1}&\le \|w(2T/3)\|_{L^1}
=\int_{I\setminus K}w(2T/3,x)\dd x+\int_{K}w(2T/3,x)\dd x\\
&\le M\,|I\setminus K|+C\delta |K|\le q. 
\end{align*}
This is the first inequality in~\eqref{1.10} with $\|w(0)\|_{L^1}=1$. 

\smallskip
{\it Step~2}. We now consider the case of arbitrary solutions $w\in\XX_T$, assuming again that $\|w(0)\|_{L^1}=1$. Let us denote by~$w_0^+$ and~$w_0^-$ the positive and negative parts of $w_0:=w(0)$, and let~$w^+$ and~$w^-$ be the solutions of~\eqref{A.31} issued from~$w_0^+$ and~$w_0^-$, respectively. Thus, we have 
$$
w_0=w_0^+-w_0^-, \quad 
\|w_0^+\|_{L^1}+\|w_0^-\|_{L^1}=1, \quad 
w=w^+-w^-. 
$$
Let us set $r:=\|w_0^+\|_{L^1}$ and assume without loss of generality that $r\ge1/2$. In view of the maximum principle for linear parabolic equations (see Section~3.2 in~\cite{landis1998}), the functions~$w^+$ and~$w^-$ are non-negative, and therefore the property established in Step~1 is true for them.  If $\|w^+(T)\|_{L^1}\le r/2$, then the contraction of the $L^1$-norm of solutions of~\eqref{A.31} implies that
$$
\|w(T)\|_{L^1}\le \|w^+(T)\|_{L^1}+\|w^-(T)\|_{L^1}\le r/2+(1-r)\le 3/4. 
$$
This coincides with the first inequality in~\eqref{1.10} with $\|w(0)\|_{L^1}=1$.

Suppose now that $\|w^+(T)\|_{L^1}> r/2$. Using the property of Step~1 with $q=\frac12$, we find $\delta_1>0$ such that 
\begin{equation} \label{2.36}
\inf_{x\in Q}w^+(T,x)\ge\delta_1 r. 
\end{equation}
Set $\e=\frac14\delta_1|I'|$ and assume that $\|w(T)\|_{L^1(I')}<\e$ (in the opposite case, the second inequality in~\eqref{1.10} holds), so that
$$
\|w^+(T)\|_{L^1(I')}-\|w^-(T)\|_{L^1(I')}<\e.
$$
It follows that
$$
\|w^-(T)\|_{L^1}\ge \|w^-(T)\|_{L^1(I')}\ge \|w^+(T)\|_{L^1(I')}-\e
\ge \delta_1 r |I'|-\frac{\delta_1}{4}|I'| \ge\e. 
$$
By the $L^1$-contraction for~$w^-$, we see that $\|w_0^-\|_{L^1}=1-r\ge \e$.  Repeating the argument applied above to~$w^+$, we can prove that if 
\begin{equation} \label{2.37}
\|w^-(T)\|_{L^1}\le \frac12(1-r),
\end{equation}
then $\|w(T)\|_{L^1}\le 1-\frac\e2$, so that the first inequality in~\eqref{1.10} holds with $q=1-\frac\e2$. Thus, it remains to consider the case when~\eqref{2.37} does not hold.  Applying the property of Step~1 to~$w^-$, we find~$\delta_2>0$ such that 
\begin{equation} \label{2.38}
\inf_{x\in I'}w^-(T,x)\ge\delta_2 (1-r). 
\end{equation}
Since $\frac12\le r\le 1-\e$, the right-hand sides in~\eqref{2.36} and~\eqref{2.38} are minorised by $\theta=\min\{\tfrac12\delta_1,\e\delta_2\}$. Denoting by~$\chi_{I'}$ the indicator function of~$I'$, we write
\begin{align*}
\|w(T)\|_{L^1}&=\int_I|w^+(T,x)-w^-(T,x)|\,\dd x\\
&=\int_I\bigl|(w^+(T,x)-\theta \chi_{I'}(x))-(w^-(T,x)-\theta \chi_{I'}(x))\bigr|\,\dd x\\
&\le\int_I\bigl(w^+(T,x)-\theta \chi_{I'}(x)\bigr)\,\dd x
+\int_I\bigl(w^-(T,x)-\theta \chi_{I'}(x))\bigr)\,\dd x\\
&=\|w^+(T)\|_{L^1}+\|w^-(T)\|_{L^1}-2\theta |I'|.
\end{align*}
In view of the $L^1$-contraction for~$w^+$ and~$w^-$, the right-hand side of this inequality does not exceed
$$
 \|w_0^+\|_{L^1}+\|w_0^-\|_{L^1}-2\theta |I'|=1-2\theta |I'|. 
$$
Setting $q=\max\{\frac34,1-\frac{\e}{2},1-2\theta |I'|\}$, we conclude that one of the inequalities~\eqref{1.10} holds for~$w$. Thus, to complete the proof of Proposition~\ref{p1.4}, it only remains to establish Lemma~\ref{l3.1}.

\begin{proof}[Proof of Lemma~\ref{l3.1}]
By the maximum principle and regularity of solutions for linear parabolic equations, it suffices to prove that
\begin{equation} \label{3.15}
\|w(\tau)\|_{L^\infty(I)}\le C_1\|w(0)\|_{L^1(I)},
\end{equation}
where $C_1>0$ does not depend on~$w$. To this end, along with~\eqref{A.31}, let us consider the dual equation
\begin{equation} \label{A.32}
\p_t z+\nu\p_x^2z+a(t,x)\p_xz=0, 
\end{equation}
supplemented with the initial condition
\begin{equation} \label{3.17}
z(T,x)=z_0(x). 
\end{equation}
Let us denote by $G(t,x,y)$ the Green function of the Dirichlet problem for~\eqref{A.32}, \eqref{3.17}. By Theorem~16.3 in~\cite[Chapter~IV]{LSU1968}, one can find positive numbers~$C_2$ and~$C_3$ depending only on~$\rho$, $s$, and~$T$ such that 
\begin{equation*} 
|G(t,x,y)|\le C_2(T-t)^{-1/2}\exp\bigl(-C_3\tfrac{(x-y)^2}{T-t}\bigr)
\quad\mbox{for $x,y\in I$, $t\in[0,T)$}. 
\end{equation*}
It follows that, for $z_0\in L^2(I)$, the solution $z\in\XX_T$ of problem~\eqref{A.32}, \eqref{3.17} satisfies the inequality
\begin{equation} \label{3.19}
\|z(0)\|_{L^\infty}\le C_4\|z_0\|_{L^1},
\end{equation}
where $C_4>0$ does not depend on~$z_0$. 

Now let $u\in\XX_T$ be a solution of~\eqref{A.31}. Taking any $z_0\in L^2(I)$ and denoting by~$z\in\XX_T$ the solution of~\eqref{A.32}, \eqref{3.17}, we write
\begin{equation} \label{3.18}
\frac{\dd}{\dd t}\bigl(w(t),z(t)\bigr)=(\p_tw,z)+(w,\p_tz)=0,
\end{equation}
where $(\cdot,\cdot)$ denotes the scalar product in~$L^2(I)$. Integrating in time and using~\eqref{3.19}, we obtain
$$
\int_Iw(T)z_0\dd x=\int_Iw(0)z(0)\dd x\le \|w(0)\|_{L^1}\|z(0)\|_{L^\infty}
\le C_4 \|w(0)\|_{L^1}\|z_0\|_{L^1}. 
$$
Taking the supremum over all $z_0\in L^2$ with $\|z_0\|_{L^1}\le 1$, we arrive at the required inequality~\eqref{3.15}. 
\end{proof}

\subsection{Completion of the proof}
\label{s3.3}
We need to prove inequalities~\eqref{1.4} and~\eqref{1.5}, as well as the piecewise continuity of $\zeta:\R_+\to H^1(I)$ and the estimate
\begin{equation} \label{3.21}
\|\zeta(t)\|_1\le C_1 e^{-\gamma t}
\min\bigl(\|u_0-\hat u_0\|_{L^1}^{2/5},1\bigr), \quad t\ge0.
\end{equation}

{\it Proof of~\eqref{1.4}}. 
The estimate for $\hat u(t)=\RR_t(\hat u_0,h)$ follows from Remark~\ref{r3.1}. Setting $t_k=2k$, we now use induction on~$k\ge0$ to prove that $u(t)=\RR_t(u_0,h+\zeta)$ is bounded on $[t_k,t_{k+1}]$ by a universal constant and that $u(t_{k+1})\in B_V(R)$, provided that $u(t_k)\in B_V(R)$. Indeed, it follows from~\eqref{1.13} that 
$$
\sup_{t_k\le t\le s_k}\|u(t)\|_2\le C_2\sup_{t_k\le t\le s_k}
\bigl(\|\hat u(t)\|_2+\|v(t)\|_2\bigr), 
$$
where $s_k=2k+1$. In view of Remark~\ref{r3.1}, the right-hand side of this inequality does not exceed a constant~$C_3(R)$. Furthermore, recalling~\eqref{1.16} and using Remark~\ref{r3.1} and inequality~\eqref{1.3} with $T=1$, we see that 
$$
\sup_{s_k\le t\le t_{k+1}}\|u(t)\|_2\le C_3(R), \quad \|u(t_{k+1})\|_2\le R.
$$
This completes the induction step. 

\smallskip
{\it Proof of~\eqref{1.5}}. 
In view of~\eqref{1.15}, it suffices to establish~\eqref{1.14} for any even integer $k\ge0$.  It follows from~\eqref{1.13}, \eqref{1.12}, and the definition of~$\chi$ that 
\begin{equation} \label{3.22}
\|u(k+1)-\hat u(k+1)\|_{L^1}=\int_I\chi_0(x)|v(k+1)-\hat u(k+1)|\,\dd x.
\end{equation}
We know that the norms of the functions~$v$ and~$\hat u$ are bounded in $L^\infty([k,k+1],H^2)$ by a constant depending only on~$R$. Since they satisfy Eq.~\eqref{1} with~$\zeta\equiv0$, we see that~$\p_tv$ and~$\p_t\hat u$ are bounded in $L^\infty([k,k+1],L^2)$ by a number depending on~$R$. By interpolation and the continuous embedding $H^1(I)\subset C^{1/2}(I)$, we see that 
$$
\|v\|_{C^{1/2}([k,k+1]\times I)}+\|\hat u\|_{C^{1/2}([k,k+1]\times I)}\le C_4(R).
$$
Since the difference $w=v-\hat u$ satisfies Eq.~\eqref{A.31} with $a=\frac12(v+\hat u)$, we conclude that Proposition~\ref{p1.4} is applicable to~$w$. Thus, we have one of the inequalities~\eqref{1.10}. If the first of them is true, then it follows from~\eqref{3.22} that~\eqref{1.14} holds with $\theta=q$. If the second inequality is true, then using~\eqref{3.22}, the contraction of the $L^1$-norm for~$w$, and relations~\eqref{1.11}, we derive 
\begin{align*}
\|u(k+1)-\hat u(k+1)\|_{L^1}&\le \|w(k+1)\|_{L^1}-\|w(k+1)\|_{L^1(I')}
\le (1-\e)\|w(0)\|_{L^1},
\end{align*}
and, hence, we obtain~\eqref{1.14} with $\theta=1-\e$. 

\smallskip
{\it Proof of the properties of~$\zeta$}. 
In view of~\eqref{1.16}, on any interval $[k,k+1]$ with odd~$k\ge0$, the function~$u$ satisfies~\eqref{1} with~$\zeta\equiv0$, and the required properties of~$\zeta$ are trivial. Let us consider the case of an even~$k\ge0$. A direct calculation show that
\begin{align*}
\zeta(t,x)&=\p_tu-\nu\p_x^2u+u\p_xu-h\\
&=-\bigl(\chi_k(1-\chi_k)w+2\nu\p_x\chi_k\bigr)\p_xw
+\bigl(\p_t\chi_k-\nu\p_x^2\chi_k+\hat u\p_x\chi_k+\chi_k w\p_x\chi_k\bigr)w,
\end{align*}
where $\chi_k(t,x)=\chi(t-k,x)$. 
Since $\chi(t,x)=1$ for $x\notin[a,b]$ and for $t\le\frac12$, we have 
$\supp\zeta\subset[k+\frac12,k+1]\times [a,b]$. By Proposition~\ref{p3.1}, $v$ and~$\hat u$ are $V$-valued continuous functions, whence we conclude that~$\zeta$ is continuous in time with range in~$H_0^1$. Moreover, since the $H^2$-norms of~$v$ and~$\hat u$ are bounded by a number depending only on~$R$, for $t\in [k,k+1]$ we have
\begin{equation} \label{3.23}
\|\zeta(t)\|_1\le C_5(R)I_{[k+1/2,k]}(t)\|w(t)\|_2\le C_6(R)\|v(k)-\hat u(k)\|_1,
\end{equation}
where $I_{[k+1/2,k]}(t)$ is the indicator function of the interval~$[k+1/2,k]$, and we used the fact that the resolving operator for the Burgers equation is uniformly  Lipschitz continuous from any ball of~$H_0^1$ to~$H^2$ for positive times; see Remark~\ref{r3.1}. Since $v(k)-\hat u(k)=u(k)-\hat u(k)$, it follows from~\eqref{1.5} and~\eqref{3.23} that~\eqref{3.21} holds. 
This completes the proof of Theorem~\ref{t2.1}.

\subsection{Absence of global approximate controllability}
\label{s3.4}
We shall prove that if $u(t,x)$ is a solution of~\eqref{1}, \eqref{2} on the interval~$[0,T]$ with some control $\zeta\in L^2(J_T\times I)$ supported by $J_T\times[a,b]$, then the restriction of~$u(T,\cdot)$ to any closed interval included in~$[0,a)$ satisfies an a priori estimate in the~$L^\infty$ norm independent of~$u_0$ and~$\zeta$. Namely, we claim that, for any positive numbers~$T_0$ and $\delta<a$, there is~$\rho>0$ such that, if $T\ge T_0$, $u_0\in L^2(I)$, and $\zeta\in L^2(J_T\times I)$ vanishes on~$J_T\times (0,a)$, then
\begin{equation} \label{3.51}
\|R_T(u_0,h+\zeta)\|_{L^\infty(K_\delta)}\le\rho,
\end{equation}
where $K_\delta=[0,\delta]$. If this is proved, then for any $R>0$ we can take~$\hat u\in L^2(I)$ such that
\begin{equation*} 
\hat u(x)\ge\rho+\delta^{1/2}R\quad\mbox{for $x\in(0,\delta)$},
\end{equation*}
and it is straightforward to check that 
$$
\|R_T(u_0,h+\zeta)-\hat u\|^2\ge\int_0^{\delta}|R_T(u_0,h+\zeta)-\hat u|^2\dd x\ge R^2.
$$

We now prove~\eqref{3.51}. In view of the regularising property of the resolving operator (see Proposition~\ref{p3.1} and the remark following it), there is no loss of generality in assuming that $u_0\in V$. In this case, if $\zeta\in L^2(J_T\times I)$, then $u(t,x)$ is continuous on~$J_T\times\bar I$. Given $\e\in(0,1)$, we fix a number~$A_\e$ (which will be chosen below) and define the function
$$
u_\e(t,x)=\frac{A_\e}{(t+\e)(a-x+\e)}.
$$
We claim that, for an appropriate choice of~$A_\e$, the function~$u_\e$ is a super-solution for~\eqref{1} in the domain $J_T\times K_a$. Indeed, let
$$
L=\max_{x\in I_a}|u_0(x)|, \quad N=\max_{t\in J_T}|u(t,a)|, \quad 
A_\e=\Lambda+\e\bigl(L(a+\e)+N(T+\e)\bigr),
$$
where $\Lambda>0$ is a large parameter that will be chosen below. For $x\in K_a$ and $t\in J_T$, we have
\begin{equation} \label{3.55}
u_\e(0,x)\ge\frac{A_\e}{\e(a+\e)}\ge L, \quad u_\e(t,0)\ge 0, \quad 
u_\e(t,a)\ge\frac{A_\e}{\e(T+\e)}\ge N. 
\end{equation}
Furthermore, a simple calculation shows that
\begin{align}
\p_tu_\e-\nu\p_x^2u_\e+u_\e\p_xu_\e
&=\frac{A_\e}{(t+\e)^2(a-x+\e)^3}\bigl(-(a-x+\e)^2-2\nu(t+\e)+A_\e\bigr)\notag\\
&\ge \frac{A_\e^2}{2(T+1)^2(a+1)^3},  \label{3.52}
\end{align}
provided that 
\begin{equation} \label{3.53}
\Lambda\ge 4\nu(T+1)+2(a+1)^2
\end{equation}
It follows from~\eqref{3.52} that if 
\begin{equation} \label{3.54}
A_\e^2\ge2(T+1)^2(a+1)^3\|h\|_{L^\infty},
\end{equation}
then~$u_\e$ is a super-solution for~\eqref{1} on the domain~$J_T\times K_a$. Inequalities~\eqref{3.53} and~\eqref{3.54} will be satisfied if we choose
$\Lambda=C(T+1)(\|h\|_{L^\infty}^{1/2}+1)$, 
where~$C>0$ is sufficiently large and depends only on~$a$ and~$\nu$. Recalling the definition of~$u_\e$, we see that the function
$$
u_\e(t,x)=\frac{C(T+1)(\|h\|_{L^\infty}^{1/2}+1)+\e\bigl(L(a+\e)+N(T+\e)\bigr)}{(t+\e)(a-x+\e)}
$$
is a super-solution for~\eqref{1} on the domain $J_T\times K_a$. It follows from~\eqref{3.55} that Proposition~\ref{p3.2} is applicable to the pair $(u_\e,u)$. In particular, we can conclude that $u(T,x)\le u_\e(T,x)$ for $x\in K_\delta$. Passing to the limit as $\e\to0$, we obtain
$$
u(T,x)\le \frac{C(T+1)(\|h\|_{L^\infty}^{1/2}+1)}{T(a-\delta)}
\quad\mbox{for $x\in K_\delta$}. 
$$
This implies the required inequality~\eqref{3.51} in which 
$$
\rho=C(a-\delta)^{-1}(1+T_0^{-1})(\|h\|_{L^\infty}^{1/2}+1). 
$$
We have thus established assertion~(b) of the Main Theorem of the Introduction.

\section{Appendix: proofs of some auxiliary assertions}

\subsection{Proof of Proposition~\ref{p3.1}}
\label{A1}
The existence and uniqueness of a solution $u\in\XX$ is well known in more complicated situations; see Chapter~15 in~\cite{taylor1996}. We thus confine ourselves to outlining the proofs of the $L^\infty$~bound and regularity. 

\smallskip
The solution~$u(t,x)$ of~\eqref{1}, \eqref{2} can be regarded as the solution of the linear parabolic equation
\begin{equation} \label{A.1}
\p_tu-\nu\p_x^2u+b(t,x)\p_xu=h(t,x),
\end{equation}
where $b\in L_{\mathrm{loc}}^2(\R_+,H_0^1)$ coincides with~$u$. If $b$, $h$, and~$u_0$ were regular functions, then the classical maximum principle would imply that (see Section~3.2 in~\cite{landis1998})
\begin{equation} \label{A.2}
|u(t,x)|\le \|u_0\|_{L^\infty}+t\,\|h\|_{L^\infty(J_t\times I)}\quad
\mbox{for all $(t,x)\in \R_+\times I$}.
\end{equation}
To deal with the general case, it suffices to approximate~$u_0$ and~$h$ by smooth functions and to pass to the (weak) limit in inequality~\eqref{A.2} written for approximate solutions. This argument shows that the inequality in~\eqref{A.2} is valid almost everywhere for any solution~$u$. 

\smallskip
We now turn to the regularity of solutions. The function~$u\in\XX$ is the solution of the linear equation
$$
\p_tu-\nu\p_x^2u=f(t,x),
$$
where the right-hand side $f=h-u\p_xu$ belongs to $L_{\mathrm{loc}}^2(\R_+,L^2)$. By standard estimates for the heat equation, we see that 
\begin{equation} \label{A.4}
u\in L_{\mathrm{loc}}^2(\R_+,H^2)\cap W_{\mathrm{loc}}^{1,2}(\R_+,L^2). 
\end{equation}
Differentiating~\eqref{1} with respect to time and setting $v=\p_tu$, we see that~$v$ satisfies the equations
\begin{equation} \label{A.3}
\p_tv-\nu\p_x^2v+v\p_xu+u\p_xv=\p_th, \quad v(0)=v_0,
\end{equation}
where $v_0=h(0)-u_0\p_xu_0+\nu\p_x^2u_0\in L^2$. Taking the scalar product of the first equation in~\eqref{A.3} and carrying out some simple transformations, we conclude that $v\in\XX$. On the other hand, it follows from~\eqref{1} that 
$$
\p_x^2u=v+u\p_xu-h\in L_{\mathrm{loc}}^2(\R_+,H^1),
$$
whence we see that $u\in L_{\mathrm{loc}}^2(\R_+,H^3)$. Combining this  with the inclusion $\p_tu\in\XX$, we obtain~\eqref{3.1}. 

\subsection{Proof of Proposition~\ref{p3.2}}
\label{A2}
Without loss of generality, we can assume that $t=T$. Define 
$$
u=u^--u^+,\quad \psi_\delta(z)=1\wedge\bigl((z/\delta)\vee 0\bigr),
$$
where $\delta>0$ is a small parameter, and $a\wedge b$ ($a\vee b$) denotes the minimum (respectively, maximum) of the real numbers~$a$ and~$b$. In view of inequality~\eqref{2.2} and its analogue for sub-solutions, the function~$u$ is non-positive almost everywhere for $t=0$ and satisfies the inequality
\begin{equation} \label{A.5}
\int_0^T(\p_tu,\varphi)\,\dd t
+\nu\int_0^T(\p_xu,\p_x\varphi)\,\dd t
-\frac12\int_0^T(w,\p_x\varphi)\dd t\le 0,
\end{equation}
where $w=(u^-)^2-(u^+)^2$, and $\varphi\in L^\infty(J_T,L^2)\cap L^2(J_T,H_0^1)$ is an arbitrary non-negative function. Let us take $\varphi(t,x)=\psi_\delta(u(t,x))$ in~\eqref{A.5}. It is easy to check that 
\begin{align*}
\int_0^T(\p_tu,\varphi)\,\dd t &=\int_I\Psi_\delta(u(T))\,\dd x,\\
\int_0^T(\p_xu,\p_x\varphi)\,\dd t&=\int_0^T\int_I|\p_xu|^2\psi_\delta'(u)\,\dd x\dd t
\ge0,\\
\biggl|\int_0^T(w,\p_x\varphi)\dd t\biggr|
&\le \int_0^T\int_I|u|\,|u^++u^-|\,|\p_xu|\psi_\delta'(u)\,\dd x\dd t\\
&\le \int_0^T\int_I\Bigl(\nu|\p_xu|^2+\frac{1}{4\nu}|u|^2\,|u^++u^-|^2\Bigr)
\psi_\delta'(u)\,\dd x\dd t,
\end{align*}
where $\Psi_\delta(z)=\int_0^z\psi_\delta(r)\dd r$. Substituting these relations into~\eqref{A.5}, we derive
\begin{align*}
\int_I\Psi_\delta(u(T))\,\dd x
&\le \frac{1}{8\nu}\int_0^T\int_I|u|^2\,|u^++u^-|^2\psi_\delta'(u)\,\dd x\dd t\\
&\le \frac{\delta}{8\nu}\int_0^T\int_I|u^++u^-|^2\,\dd x\dd t
\le \frac{\delta}{8\nu}\,\bigl\|u^++u^-\bigr\|_{L^2(J_T\times I)}^2,
\end{align*}
where we used the fact that $0\le u\le \delta$ on the support of $\psi_\delta'(u)$.
Passing to the limit as $\delta\to0^+$, we derive 
$$
\int_I\bigl(u(T)\vee0\bigr)\,\dd x\le 0.
$$ 
This inequality implies that $u(T,x)\le0$ for a.e.~$x\in I$, which is equivalent to~\eqref{2.3}.

\subsection{Proof of Proposition~\ref{p3.3}}
\label{A3}
We apply an argument similar to that used in the proof of Lemma~\ref{l3.1}; see Section~\ref{s3.2}. 
Let us note that the difference $w=u-v\in\XX$ satisfies the linear equation~\eqref{A.31}, in which $a=\frac12(u+v)$. Along with~\eqref{A.31}, let us consider the dual equation~\eqref{A.32}. 
The following result is a particular case of the classical maximum principle. Its proof is given in Section~III.2 of~\cite{landis1998} for regular functions~$a(t,x)$ and can be obtained by a simple approximation argument in the general case. 

\begin{lemma} \label{lA.1}
Let $a\in L^2(J_T,H^1)$ for some  $T>0$. Then, for any $z_0\in L^2(I)$, problem~\eqref{A.32}, \eqref{3.17} has a unique solution $z\in\XX_T$. Moreover, if~$z_0\in L^\infty(I)$, then~$z(t)$ belongs to~$L^\infty(I)$ for any $t\in J_T$ and satisfies the inequality
\begin{equation} \label{A.33}
\|z(t)\|_{L^\infty}\le \|z_0\|_{L^\infty}\quad\mbox{for $t\in J_T$}. 
\end{equation}
\end{lemma}

To prove~\eqref{21}, we fix $t=T$ and assume without loss of generality that $s=0$. By duality, it suffices to show that, for any $z_0\in L^\infty(I)$ with norm $\|z_0\|_{L^\infty}\le 1$, we have
\begin{equation} \label{A.34}
\int_Iw(T)z_0\,\dd x\le \|w(0)\|_{L^1}
\end{equation}
Let $z\in\XX_T$ be the solution of~\eqref{A.32}, \eqref{3.17}. Such solution exists in view of Lemma~\ref{lA.1} and the inclusion $a\in L^2(J_T,H_0^1)$, which is ensured by the regularity hypothesis for~$u$ and~$v$. It follows from~\eqref{A.31} and~\eqref{A.32} that relation~\eqref{3.18} holds. Integrating it in time, we see that
$$
\int_Iw(T)z_0\,\dd x=\int_Iw(0)z(0)\,\dd x\le \|w(0)\|_{L^1} \|z(0)\|_{L^\infty}. 
$$
Using~\eqref{A.33} with $t=0$, we arrive at the required inequality~\eqref{A.34}. 

\begin{remark} \label{r4.2}
We have proved in fact that if $w\in \XX_T$ is a solution of the linear equation~\eqref{A.31}, in which the coefficient~$a$ belongs $L^2(J_T,H^1)$ , then $\|w(t)\|_{L^1}\le \|w(s)\|_{L^1}$ for $0\le s\le t\le T$. 
\end{remark}

\addcontentsline{toc}{section}{References}
\def\cprime{$'$} \def\cprime{$'$} \def\cprime{$'$}
  \def\polhk#1{\setbox0=\hbox{#1}{\ooalign{\hidewidth
  \lower1.5ex\hbox{`}\hidewidth\crcr\unhbox0}}}
  \def\polhk#1{\setbox0=\hbox{#1}{\ooalign{\hidewidth
  \lower1.5ex\hbox{`}\hidewidth\crcr\unhbox0}}}
  \def\polhk#1{\setbox0=\hbox{#1}{\ooalign{\hidewidth
  \lower1.5ex\hbox{`}\hidewidth\crcr\unhbox0}}} \def\cprime{$'$}
  \def\polhk#1{\setbox0=\hbox{#1}{\ooalign{\hidewidth
  \lower1.5ex\hbox{`}\hidewidth\crcr\unhbox0}}} \def\cprime{$'$}
  \def\cprime{$'$} \def\cprime{$'$} \def\cprime{$'$}
\providecommand{\bysame}{\leavevmode\hbox to3em{\hrulefill}\thinspace}
\providecommand{\MR}{\relax\ifhmode\unskip\space\fi MR }
\providecommand{\MRhref}[2]{%
  \href{http://www.ams.org/mathscinet-getitem?mr=#1}{#2}
}
\providecommand{\href}[2]{#2}

\end{document}